\documentclass[11pt,amsfonts]{amsart}
\usepackage{amsmath,amsfonts,amsthm,amssymb,color}
\usepackage[T1]{fontenc}
\usepackage{enumerate}
\usepackage{graphicx}
\usepackage{float}
\usepackage[ruled,vlined]{algorithm2e}
\usepackage{amssymb}
\usepackage{amsmath}
\usepackage{color}
\usepackage{layout}
\usepackage{amsfonts}
\usepackage{dsfont}
\usepackage{cite}
\usepackage{color}
\usepackage{listings}
\usepackage{anysize}
\usepackage[noend]{algpseudocode}
\usepackage{ulem}
\usepackage{latexsym}

\usepackage{caption}
\usepackage{subcaption}

\usepackage{pdfsync}

\newtheorem{lemma}{Lemma}
\newtheorem{definition}{Definition}

\newtheorem{theorem}{Theorem}
\newtheorem{remark}{Remark}

\def\real{{\mathord{{\rm I\kern-2.8pt R}}}}        
\def\inte{{\mathord{{\rm I\kern-2.8pt N}}}}

\def\sZZ{{\rm Z\kern-2.8ptem{}Z}}

\def\z{{\mathchoice
  {\sZZ}
  {\sZZ}
  {\rm Z\kern-0.30em{}Z}
  {\rm Z\kern-0.25em{}Z} }}
\def\sQQ{{\kern 0.27em \vrule height1.45ex width0.03em depth0em
          \kern-0.30em \rm Q}}
\def\qu{{\mathchoice
    {\sQQ}
    {\sQQ}
  {\kern 0.225em \vrule height1.05ex width0.025em depth0em \kern-0.25em \rm Q}
  {\kern 0.180em \vrule height0.78ex width0.020em depth0em \kern-0.20em \rm Q}
        }}
\def\sCC{{\kern 0.27em \vrule height1.45ex width0.03em depth0em
          \kern-0.30em \rm C}}
\def\complex{{\mathchoice
    {\sCC}
    {\sCC}
  {\kern 0.225em \vrule height1.05ex width0.025em depth0em \kern-0.25em \rm C}
  {\kern 0.180em \vrule height0.78ex width0.020em depth0em \kern-0.20em \rm C}
        }}

\newcommand{\ba}{\begin{array}}
\newcommand{\ea}{\end{array}}
\newcommand{\be}{\begin{equation}}
\newcommand{\ee}{\end{equation}}
\newcommand{\bea}{\begin{eqnarray}}
\newcommand{\eea}{\end{eqnarray}}
\newcommand{\beaa}{\begin{eqnarray*}}
\newcommand{\eeaa}{\end{eqnarray*}}

\def\b{\beta}

\def\z{\zeta}

\font\tenmath=msbm10 \font\sevenmath=msbm7 \font\fivemath=msbm5
\newfam\mathfam \textfont\mathfam=\tenmath
\scriptfont\mathfam=\sevenmath \scriptscriptfont\mathfam=\fivemath

\def \b{\noindent}

\def \={{\buildrel {\rm (law)} \over =}}

\def\qed{ \hfill \vrule width.25cm height.25cm depth0cm\smallskip}

\newcommand{\basa}{\begin{assumption}}
\newcommand{\easa}{\end{assumption}}

\newcommand{\bas}{\begin{assum}}
\newcommand{\eas}{\end{assum}}

\def\limsup{\mathop{\overline{\rm lim}}}

\newcommand{\ignore}[1]{}
\begin{document}
\title[Explosion time in SDE driven by fBm]{On explosion time in stochastic differential equations driven by fractional Brownian motion}

\author[J. Garz\'on]{Johanna Garz\'on}
\address{Universidad Nacional de Colombia, Departamento de Matem\'aticas, Colombia}
\email{mjgarzonm@unal.edu.co}

\author[J.A. Le\'on]{Jorge A. Le\'on}
\address{Cinvestav-IPN, Departamento de Control Autom\'atico, Mexico}
\email{jleon@ctrl.cinvestav.mx; joleon@cinvestav.mx }

\author[S Torres]{Soledad Torres} 
\address{Universidad de Valparaiso, CIMFAV, and Universidad Central de Chile,  Chile}
\email{soledad.torres@uv.cl}

\author[C.A. Tudor]{Ciprian A. Tudor}
\address{Universit\'e de Lille 1, UFR Math\'ematiques, France}
\email{ciprian.tudor@univ-lille.fr}

\author[L. Viitasaari]{Lauri Viitasaari}
\address{Aalto University School of Business, Department of Information and Service Management, Finland}
\email{lauri.viitasaari@iki.fi}


\begin{abstract}
In this article, we study the explosion time of the solution to autonomous stochastic differential equations driven by the fractional Brownian motion with Hurst parameter $H>1/2$. With the help of the Lamperti transformation, we are able to tackle the case of non-constant diffusion coefficients not covered in the literature. In addition, we provide an adaptive Euler-type numerical scheme for approximating the explosion time.
\end{abstract}
\maketitle
\vskip0.3cm
{\bf 2020 AMS Classification Numbers: 65C30, 65L20, 60H10, 60G22} 

\section{Introduction}
Fractional Brownian motion (fBm), with Hurst parameter $H$, denoted by $B^{H}$, is a centered Gaussian process with covariance $R_{H}(t,s)= \mathbf{E}(B^{H}({t}) B^{H}({s}))= \frac{1}{2}\left( t^{2H}+ s^{2H} -\vert t-s\vert ^{2H}\right)$ for $t,s\in [0,T]$ and $H\in (0,1)$.  In the case $H=\frac{1}{2}$ the fBm reduces to the standard Wiener process. Analysis of the fBm, stochastic integration concerning it, and stochastic differential equations (SDEs) driven by it; fBm has constituted an important research direction in probability theory during the last few decades. This is due to the fact that fBm is perhaps the simplest generalization of the standard Brownian motion having many wanted features such as Gaussianity, self-similarity, and stationarity of the increments. Contrary to the standard Brownian motion, however, fBm does not have independent increments, allowing its usage in modeling memory effects.
Let us consider the following scalar fractional SDE
\begin{equation}
\label{1}
X({t})= x({0}) + \int_{0} ^{t} b(X(s)) ds + \int_{0} ^{t} \sigma (X(s)) dB^{H}({s}), \hskip0.5cm t\in [0,T],
\end{equation}
where $X(0)= x(0) >0$, the coefficients $b$ and $\sigma$ satisfy some regularity conditions (detailed later on in Section \ref{sec:explosion}),  and $(B^{H}_{t}) _{t\in [0,T]} $ is a fractional Brownian motion.  In the case of standard Brownian motion ($H=\frac12$) it is well-known that the solution of the stochastic equation (\ref{1}) may explode in a finite time when the coefficient $b$ does not satisfy the linear growth assumption. That is, the trajectories of the solution to (\ref{1}) may approach infinity as $t$ approaches a certain finite time $T_{e}$. As the SDE is random, the explosion time $T_{e}$ is also random and depends on the particular path. Whether such explosion happens for a given coefficient function $b$ is completely understood and characterized, see \cite{KS} or \cite{RW} for a detailed presentation. 

The motivation for the study of exploding SDEs comes from applications. Indeed, SDEs with explosions are applied, for example, in stochastic models with fatigue cracking of materials, where the explosion time corresponds to the time of ultimate damage or fatigue failure in the material (see \cite{groisman, zhang} and the references therein). On the other hand, in \cite{zheng}, the authors claim that modeling this process as a Markov diffusion process requires a short correlation time of the noise term, which may not be realistic in some situations. 
Therefore, it is natural to take a Gaussian long memory process as the noise with Hurst parameter $H >1/2$, suggesting that a fractional SDE will be a good model for this type of problem. 

There are only a few results concerning fractional SDEs with explosion times. In \cite{leon}, the authors provided necessary and sufficient conditions for the solution to \eqref{1} explode in finite time in the case of constant diffusion $\sigma$. More precisely, in \cite{leon}, it was proven that if $b$ is a positive, locally Lipschitz, and non-decreasing function and $\sigma=1$, then the solution to \eqref{1} explodes in finite time with probability $1$ if and only if
$\int_{x(0)}^{\infty}\frac{1}{b(s)}ds < \infty.$  In a related work, \cite{XZL}  the authors studied the existence and  uniqueness of solutions  to the fractional SDE with non-Lipschitz coefficients. They provided criteria for explosion under the assumption that the diffusion is close to zero (in a certain sense).  Finally, we mention \cite{narita} where sufficient conditions for the nonoccurrence of explosion for the  solution to (\ref{1}) are given in the case that $H>\frac{1}{2}$ and the stochastic integral with respect to fBm is understood in the Wick (non-pathwise) sense. In \cite{narita}, the author also  gives examples of explosive solutions. For concrete examples of explosion for pathwise fractional SDEs, see also \cite{HN} or \cite{GLT}. However, to the best of our knowledge, there does not exist a general criterion for the explosion in the case of fractional pathwise SDEs under a more general diffusion coefficient $\sigma$.

Analysis of SDEs \eqref{1} also requires numerical methods as, in practice, one cannot simulate/solve \eqref{1} in continuous time. Numerical solutions based on the discretization of stochastic differential equations (SDEs) driven by the classical  Brownian motion or, more generally,  by semimartingales have been widely analyzed and are already well-understood. For details on various situations and techniques, we refer to the monograph \cite{KP} (see also \cite{KoPr}). For results on various Euler, Milstein, or Crank-Nicholson  schemes in the case of fractional SDEs, see e.g. \cite{GA,GN,mishura,No,NoNe}, among others. However, it is worth to emphasizing that the development of numerical schemes for SDEs with an explosion is much more complicated. Indeed, one cannot use simple deterministic discretizations because when one approaches the explosion time, the time steps must be adapted to the sudden rapid growth of the solution. 

The main contribution of this article is a method to approximate the explosion time of solution to \eqref{1}, where $H>\frac12$ and the stochastic integral is understood in the pathwise sense. We begin by providing criteria for the explosion in the case of a more general diffusion coefficient $\sigma$ than constant (studied in \cite{leon}) or small (studied in \cite{XZL}). To do this, we apply the Lamperti transform, allowing us to treat $\sigma$ as constant, after which the Osgood criterion for explosion provided in \cite{leon} is at our disposal. For the approximation scheme, we follow the ideas of \cite{acosta} and \cite{groisman} to construct an adaptive Euler scheme that reproduces the explosion of the solution. For this purpose, however, we do not have a martingale structure at our disposal, so we need to prove certain key technical estimates via other means. We also present a simple algorithm for approximating of the explosion time based on a prediction formula for the fBm derived recently in \cite{So-Vi}.

The rest of the article is organized as follows: Section \ref{sec:explosion} recalls some preliminaries on fractional SDEs \eqref{1} and provides criteria for explosion. In Section \ref{sec:scheme}, we introduce an adaptive Euler scheme and present our main theorem on its convergence, cf. Theorem \ref{Th aproxY}. In Section \ref{sec:numerical}, we illustrate our method with numerical examples. We end the paper with an appendix containing the proof of Theorem \ref{Th aproxY}.

\section{Fractional SDEs with explosion}
\label{sec:explosion}

In Sections \ref{sec:fractional-calculus}-\ref{sec:SDE}, we recall some preliminaries on fractional calculus and corresponding integration with respect to fBm that allow us to study fractional SDEs. We continue with Section \ref{sec:explosion-criteria}, where we recall the explosion criteria of \cite{leon} that, when combined with the Lamperti transform, gives us our generalized explosion criteria, Theorem \ref{teo_explosion}. Interesting examples are gathered in section \ref{sec:examples}. For details on our preliminaries, see e.g. \cite{KS,nualart2,Za}.


\subsection{Elements of the fractional calculus}
\label{sec:fractional-calculus}
The fractional integration theory below is valid for H\"older continuous functions and even beyond, see, e.g. \cite{HTV}, but for our purposes, we will restrict ourselves to the case of the fBm. Let now $(B^{H}({t})) _{t \ge 0}$   be an fBm with Hurst parameter $H\in (\frac{1}{2}, 1)$. For notational simplicity, here and in the sequel, we omit the superscript $H$ and simply write $B$ instead of $B^H$ whenever confusion cannot arise. We begin by recalling the H\"older continuity of $B$.
\begin{lemma}
\label{tckolmogorov}
For each $T>0$ and $0<\rho<H$ there is a non-negative random variable $F_{B}=F_{B}({\omega}, \rho, T)$ such that $\mathbf{E}(F_B^p)< \infty$ for all $p\geq 1$ and
\begin{equation*}
\label{ckolmogorov}
|B(t)-B(s)|< F_B|t-s|^{H-\rho} \  \  \text{a.s.} \ \ \ t,s\in[0,T].
\end{equation*}
\end{lemma}

The pathwise integral w.r.t. fractional Brownian motion $B$  can be defined by (see \cite{Za})

\begin{equation}
    \label{eqdefintegral}
\int_a^bf_sdB(s)= \int_a^b(D_{a+}^{\alpha}f)(s)(D_{b-}^{1-\alpha}B_{b-})(s)ds,
\end{equation}
if last integral is well-defined,  where $\alpha \in (0 , 1)$ and  for  $s \in [0 , T]$,
 \begin{equation}
(D_{a+}^{\alpha}f)(s)=\frac{1}{\Gamma(1-\alpha)}\left[\frac{f(s)}{(s-a)^{\alpha}}+\alpha\int_a^s\frac{f(s)-f(u)}{(s-u)^{1+\alpha}}du\right]I_{(a,b)}(s),
\end{equation}

\begin{equation}
(D_{b-}^{1-\alpha}B_{b-})(s)=\frac{1}{\Gamma(\alpha)}\left[\frac{B_{b-}(s)}{(b-s)^{1-\alpha}}+(1-\alpha)\int_s^b\frac{B_{b-}(s)-B_{b-}(u)}{(u-s)^{2-\alpha}}du\right]I_{(a,b)}(s),
\end{equation}
are the so-called Riemann-Liouville fractional derivatives, where
$$B_{b-}(s)=(B(b)-B(s))I_{(a,b)}(s).$$

If $f\in C^\nu(a,b)$ ($\nu$-H\"older continuous function) with  $\nu+H>1$, the pathwise fractional integral exists for any $\alpha\in (1-H, \nu)$ and it admits the estimate
\begin{equation}
    \label{eqcotaintegral}
\left|\int_a^b f_sdB(s)\right|\leq F_0\left[\int_a^b\frac{|f(s)|}{(s-a)^{\alpha}}ds+\int_a^b\int_a^s\frac{|f(s)-f(u)|}{(s-u)^{\alpha+1}}duds\right],
\end{equation}
where
\begin{equation}
    \label{eqdefc0}
  F_0=F_0(b, a)=F_0(\omega, b, a)=C\cdot \sup_{a<s<b}|D_{b-}^{1-\alpha}B_{b-}(s)|.
    \end{equation}
If $a=0$, we denote $F_0(b)=F_0(b, 0)=F_0(\omega, b, 0)$.

We also recall the following result, see \cite[Lemma 7.5]{nualart2}.		

\begin{lemma} For all $T>0$, $H>1/2$ and $p\in[1, \infty)$ we have
    \label{lecotac0}
    $$\mathbf{E}(F_0^p(T))=\mathbf{E}\left(\sup_{0\leq s\leq t\leq T}D_{t-}^{1-\alpha}B_{t-}(s)\right)^p<\infty.$$
\end{lemma}
 

\subsection{The pathwise stochastic differential equation driven by fBm}
\label{sec:SDE}
We consider a fractional stochastic differential equation of the form (\ref{1}) where $B^{H}:=B$ is a fractional Brownian motion with Hurst index $H>1/2$, $x(0)>0$ and $\sigma, b: \Omega\times \mathbb{R}\to \mathbb{R} $ are measurable functions. That is, we consider the integral equation
\begin{equation*}
X(t)= x({0}) + \int_{0} ^{t} b(X(s)) ds + \int_{0} ^{t} \sigma (X(s)) dB({s}), \hskip0.5cm t\ge 0,
\end{equation*}
where the integral with respect to $B$ is defined as in \eqref{eqdefintegral}, Section \ref{sec:fractional-calculus}.
To guarantee the existence of a solution to (\ref{1}), 
 we suppose that the coefficients $\sigma$ and $b$ satisfy the following assumptions: 

\begin{equation}\label{AB}
\hspace{-9cm}\mbox{\textbf{Assumptions A, B, and C:}}
\end{equation}
\begin{description}
    \item[A] $\sigma$ is differentiable and there exist constants $C_\sigma, C_N, K_N, K >0$ and $ \kappa \in (\frac{1}{H} -1, 1]$, such that $\sigma$ satisfies:
    \begin{enumerate}
    \item  Lipschitz continuity: 
		$$|\sigma(x)-\sigma(y)|\leq C_\sigma|x-y|, \ \ x,y \in \mathbb{R}.$$
    \item Local H\"older continuity for the derivative of $\sigma$:
    $$|\sigma'(x)-\sigma'(y)|\leq C_N |x-y|^{\kappa}, \ \ |x|, |y|\leq N.$$
\end{enumerate}

\item[B]  Local Lipschitz continuity: for any $N>0$ there exists $K_N >0$ such that
    $$|b(x)-b(y)|\leq K_N |x-y|, \ \  |x|, |y|\leq N.$$  

\item[C]  Linear growth: for every $x\in \mathbb{R}$
    $$|b(x)|\leq K(|x|+1), \ \ x \in \mathbb{R}.$$
\end{description}

Let us recall \cite[Theorem 2.1]{nualart2} concerning the solution's existence and to (\ref{1}).
\begin{theorem}
\label{texistenciax}
Under conditions \textbf{A}, \textbf{B}, and \textbf{C},  the equation (\ref{1}) has a unique solution $X$ in $C^{H-\rho}(0,T)$ a.s. for any $0<\rho<H$. Moreover, there exists a positive (random) constant $F_X=F_X(\omega)$ satisfying $\mathbf{E}(F_X^p)<\infty$ for all $p\geq 1$ such that \begin{equation}
\label{ecotax}
\sup_{0\leq t \leq T}|X(t)|< F_X, \ \ \ \sup_{0\leq s\leq t \leq T}|X(t)-X(s)|< F_X|t-s|^{H-\rho}.
\end{equation}
\end{theorem}

\

If the function $b$ in the fractional SDE (\ref{1})  does not satisfy Assumption \textbf{C} on linear growth, the solution of this equation may explode in finite time. That is, there exists a  random time $T^X_e<\infty$ such that $X(t)$ is defined in
$[0, T^X_e)$ and $\displaystyle\lim_{t\uparrow T^X_e} |X(t)|= \infty$ (see McKean \cite{MK}, Section 3.3). 

To study the explosion of the solution to \eqref{1}, let $M> x(0)$ be fixed and let the functions $b_M$ and $\sigma_M$ be such that

\begin{enumerate}
	\item $b_M$ and $\sigma_M$ are bounded and satisfy the conditions \textbf{A} and \textbf{B}.
\item $b_M(x) = b(x)$ and  $\sigma_M(x) = \sigma(x)$ for any $x\in \mathbb{R}$ with $|x|\leq M$.
\end{enumerate}
By Theorem \ref{texistenciax}, there exists a unique solution $X_M$, called the local solution, to 
\begin{equation}\label{XM}
X_M(t)= x({0}) + \int_{0} ^{t} b_M(X_M(s)) ds + \int_{0} ^{t} \sigma_M (X_M(s)) dB(s), \quad t\in [0,T].
\end{equation}
For $M>x(0)$, we define  the  random time
\begin{equation}
\label{edefrm}
T^X_M:=\inf\{t\geq 0: |X_M(t)|\geq M\},
\end{equation}
where, as usual, $\inf\{\emptyset\} = \infty$. Clearly, the uniqueness of the  solution to \eqref{XM} implies that, for $M_1 < M_2$ we have $X_{M_1}(t) = X_{M_2}(t)$ on $[0; T^X_{M_1})$, and $T^X_{M_1} \leq T^X_{M_2}$ (actually, due to continuity of the solutions, the equality is possible only when both are infinite).
Let\begin{equation}
\label{edefrm-T}
T_e^X = \lim_{M\to \infty} T^X_M \ \ \text{a.s.}
\end{equation}
Note that $T_e^X\in (0 ,\infty]$ is well defined because the sequence $\{T^X_M:M>0\}$ is non-decreasing. Moreover, there exists a process $X=\{X(t): t \in [0 , T_e^X)\}$  such that
$X(t) = X_M(t)$ on $[0 , T^X_M)$ for each $M> x(0)$.
We call this process $X$ the solution to \eqref{1} in the case where $b$ and $\sigma$ satisfy Assumptions \textbf{A} and \textbf{B} but not \textbf{C}. This gives us the following definition.
\begin{definition}
    The random time $T_e^X$ given by \eqref{edefrm-T} is called the explosion time of \eqref{1}. If $T_e^X < \infty$, with positive probability we say that the solution to the SDE \eqref{1} explodes in finite time. 
\end{definition}

\subsection{Osgood criteria for the explosion}
\label{sec:explosion-criteria}
In what follows, we will give hypotheses on the
coefficients $b$ and $\sigma$, which ensure on the one hand, the use of Lamperti's transformation and, on the other hand, that they satisfy the hypotheses necessary to have a finite-time explosion a.s. as presented in references \cite{leon , leon2}.

First, let us recall the hypotheses and results obtained in \cite{leon2}. The authors consider the following equation
\begin{equation}\label{LPV}
X(t) = x(0) +  \int_0^t a(s)g(X(s))ds + h(t), \quad  t \geq 0, 
\end{equation}
under the following conditions:
\begin{description}
\item[H1] $a : (0, \infty) \to  (0, \infty)$ is a continuous function such that
$\displaystyle \lim_{t \to \infty }  \int_t^{t+\eta} a(s)ds > 0$, for some $\eta > 0$. 
\item[H2]  $g : \mathbb{R} \to  [0, \infty)$ is a continuous function and there exist $-\infty \leq l < \infty$ and $-\infty < r < \infty$ such that $g > 0$ and locally Lipschitz on $(l, \infty)$, and $g : [r, \infty) \to  (0, \infty)$ is non-decreasing. 
\item[H3] $h : [0, \infty) \to  \mathbb{R}$ is a continuous function such that $\displaystyle \limsup_{t \to \infty}\displaystyle\inf_ { 0 \leq u \leq {\tilde{\eta}}} h(t + u)= \infty $, for some $ {\tilde{\eta}} >0$, where $\limsup$ denotes the upper limit.
\end{description}
\begin{remark}
   In what follows we  suppose that $l= -\infty$. Otherwise, as it is point out in \cite{leon2} (Remark (3.4)), we only need to assume that the initial condition of \eqref{1} is large enough. 
\end{remark}
The next theorem, \cite{leon2}, provides criteria for establishing whether the solution to \eqref{LPV} explodes in finite time. 
\begin{theorem}\label{teoperalta}
    Let $x(0) \in \mathbb{R}$ and assume {\bf H1-H3}. Then the explosion time $T^X_{e}$ of the
solution of \eqref{LPV} is finite if and only if
$$\int^\infty_r \frac{ds}{g(s)} < \infty.$$ 
\end{theorem}
\begin{remark}
Here, $r$ is given in \textbf{H2}. Also, by plugging in $h(t) = \sigma B(t)$ with constant $\sigma$ one obtains an explosion criteria for the SDE \eqref{1} in the case of constant diffusion $\sigma$. 
\end{remark}
Let us return to our equation \eqref{1} with a more general function $\sigma$. We pose the following additional hypothesis.
\begin{description}
	\item[D] For the coefficient functions $b$ and $\sigma$, we assume:
 \begin{enumerate}
 \item For any $M>0$, there exist bounded functions $\sigma_M$ and $b_M$ satisfying \textbf{A} and \textbf{B} such that $b_M(x)=b(x)$ and $\sigma_M(x)=\sigma(x)$ for any $|x|\leq M$.
\item  $\sigma: \mathbb{R} \to  (0, \infty)$ is a strictly positive and continuous function. 
	\item $\dfrac{b}{\sigma}: \mathbb{R} \to  [0, \infty)$ satisfies condition 
{\bf{H2}}.
\end{enumerate}
\end{description}
Note that the first statement guarantees that the local solutions $X_M$ exists. Under \textbf{D} and with the initial condition $x(0)>0$, we define
\begin{equation}
\label{eqTheta}
    \Theta(x)= \int_{x(0)}^x \frac{ds}{\sigma(s)}, \quad x\in\mathbb{R}.
\end{equation}
Note that $\Theta: \mathbb{R} \to \mathbb{R}$ is well-defined because $\sigma$ is strictly positive and continuous. 
Since $\sigma$ is positive, it follows that $\Theta$ is also strictly increasing, and hence $\Theta^{-1}$ exists. Moreover, $\Theta^{-1}$ is strictly increasing as well, and locally Lipschitz. Now following \cite{Neuenkirch}, the process 
\begin{equation}\label{theta(Y)}
    Y(t)= \Theta(X(t)), \quad t\ge 0,
\end{equation}
is the unique solution to the stochastic differential equation 
\begin{eqnarray}\label{eq-gen}
\begin{cases}
\label{eqY}
dY(t)&= g(Y(t))dt  + dB(t), \quad t>0 \\
Y(0) &= \Theta(x(0))
\end{cases}
\end{eqnarray}
with 
\begin{equation}
    \label{eqg}
    g(x)= \frac{b(\Theta^{-1}(x))}{\sigma(\Theta^{-1}(x))}, \quad x\in \mathbb{R}.
\end{equation}
Now, equation \eqref{eqY} has the form of equation \eqref{LPV} and the fractional Brownian motion satisfies hypothesis {\bf H3}, as pointed out  in \cite{leon}. Moreover, Assumption {\bf D} implies that $g$ satisfies {\bf H2}, and {\bf H1} is trivially true for $a \equiv 1$. This leads to the following criteria.
\begin{theorem}
\label{teo_explosion}
    Let Hypothesis {\bf D} hold and  assume $\Theta(\infty)=\infty;\ \ \Theta(-\infty)=-\infty$. Then the solution $(X(t))_{t\ge 0}$ to \eqref{1} explodes in finite time, with probability one, if and only if  
 \begin{equation}
 \label{eqint_b}
     \int_{a}^{\infty} \frac{ds}{b(s)}< \infty,
 \end{equation}
 where $a = \Theta^{-1}(r)$ and $r$ is given in Statement (3) of Hypothesis \textbf{D}.
\end{theorem}
\begin{proof}
Firstly, it follows from Hypothesis \textbf{D} that the local solutions $Y_M$ (as defined in \eqref{XM}) associated with \eqref{eq-gen} exist, and that the explosion time $T_e^Y$ is well-defined. Due to Lamperti's transform $Y(t) = \Theta(X(t))$, it follows that the local solutions $X_M$ defined in \eqref{XM} to equation \eqref{1} are well-defined as well. Moreover, since $\Theta(\infty) = \infty$, $\Theta(-\infty) = -\infty$ and $\Theta$ and $\Theta^{-1}$ are continuous functions, it follows that $Y= \Theta(X)$ explodes in finite time if and only if $X =\Theta^{-1}(Y)$ explodes in finite time. Now, since $g$ is a strictly positive, locally Lipschitz, and non-decreasing function, from Theorem \ref{teoperalta} (or   \cite[Theorem 4.3]{leon}), the solution of equation \eqref{eqY} explodes in finite time with probability 1 if and only if 
\begin{equation*}
 \label{eqint_g}
    \int_{r}^{\infty} \frac{dx}{g(x)} < \infty.
\end{equation*}
The claim now follows since, by the change of variable $u=\Theta^{-1}(x)$, we have 
\begin{equation}
 \label{eqint_g1}
    \int_{r}^{\infty} \frac{dx}{g(x)}= \int_{r}^{\infty} \frac{\sigma(\Theta^{-1}(x))}{b(\Theta^{-1}(x))}dx= \int_{a}^{\infty} \frac{du}{b(u)}
\end{equation}
where $a = \Theta^{-1}(r)$, which is finite if and only if  \eqref{eqint_b} holds. This completes the proof.

\qed

\end{proof}

\begin{remark}
    It is  pointed out in \cite{leon} that if $Y$ explodes in finite time, then $\lim_{t \uparrow T_e^Y} Y_t = \infty$. Therefore, if the solution $X$ explodes in finite time, we have that $\displaystyle\lim_{t \uparrow T_e^X} X_t = \infty$.
\end{remark}

\subsection{Examples}
\label{sec:examples}
In this section, we present the following two interesting examples with non constant $\sigma$, in which cases the 
result of \cite{leon} (that assumes constant $\sigma$) cannot be applied. 

\vspace{1cm}
{\bf Example 1.}
In \cite{narita}, it is proven that the solution to the stochastic differential equation 
\begin{equation}\label{narita}
dX(t) = v(t)X^N(t) dt + a(t)X(t)dB(t); \quad  X(0) = x(0)>0, \quad t \in [0,T],
\end{equation}
where $v:[0, \infty)\to (0, \infty)$ and $a: [0,\infty) \to \mathbb{R}$ are continuous functions and $N\ge 2$, explodes in finite time. In addition, it was shown that the solution remains positive. 
Now, to fix ideas taking $v(t)=a(t)\equiv 1$ and $x(0)=1$ leads to
\begin{equation}
    \label{ejemplo1}
    dX(t) = X^{N}(t) dt +X(t)dB(t); \quad  x(0) = 1,  \quad t \in [0,T].
\end{equation}
Note that, in this case, we only need to assume that $\sigma: (0,\infty) \to (0,\infty)$ is strictly positive and continuous instead of Statement (2) in Hypothesis \textbf{D}.

Now from \eqref{eqTheta} we get, for $x\in (0, \infty)$, 
$$\Theta(x)= \ln(x) \quad \text{and} \quad \Theta^{-1}(x)=e^x. $$
Thus \eqref{eqg} gives 
$g(x)=e^{(N-1)x}$ that is a positive function, locally Lipschitz, and non-decreasing, and we see that Assumption \textbf{D} is satisfied.
Since now $\int_1^\infty \frac{ds}{s^N}< \infty$ for $N \ge 2$, Theorem \ref{teo_explosion} implies that the solution to  \eqref{ejemplo1} explodes in finite time. In fact, the process $Y=(Y(t))_{t\ge 0}$ with $ Y(t)= \Theta(X(t))$,  is the unique solution to the stochastic differential equation 
\begin{eqnarray*}
\begin{cases}
dY(t)&= e^{(N-1)Y(t)}dt  + dB(t), \quad t>0 \\
Y(0) &= 0.
\end{cases}
\end{eqnarray*}
Thus, Theorem \ref{teoperalta} implies that the process  $Y$ explodes in finite time, and hence the solution $X$ also explodes in finite time.
%
%
%

\vspace{1cm}
{\bf Example 2.} In this example, we choose $b$ and $\sigma$ as in \cite{groisman}.  That is, we set $b(x)=(|x|+0.1)^p$ and   $\sigma(x)=(|x|+0.1)^q$ with $0<q<1$ and $p > 1$. 
 Also, we consider the initial condition $x(0) >0$. It follows that $\Theta$, defined in  \eqref{eqTheta}, is given by 
$$\Theta(x)= \begin{cases}
\frac{1}{1-q} \left\{(x+0.1)^{1-q}-(x(0)+0.1)^{1-q}\right\}, & x \ge 0 \\
     \frac{1}{1-q}\left\{ 2 \times 0.1^{1-q }- (0.1-x)^{1-q}-(x(0)+0.1)^{1-q}\right\}, & x< 0,
\end{cases}$$
with an inverse

\begin{footnotesize}
$$\Theta^{-1}(x) = \begin{cases}
    \left\{(1-q)x+(x(0)+0.1)^{1-q} \right\}^{1/(1-q)}-0.1, &  x \ge \frac{1}{1-q} [0.1^{1-q} - (x(0) +0.1)^{1-q}]\\  
     0.1 - \left\{2 \times 0.1^{1-q} - (x(0) + 0.1)^{1-q} - x(1-q) \right\}^{1/(1-q)}, &  x \le \frac{1}{1-q} [0.1^{1-q} - (x(0) +0.1)^{1-q}].
\end{cases}$$
\end{footnotesize}
\vspace{0.5cm}

Now $Y=(Y(t))_{t>0}$ defined as  $ Y(t)= \Theta(X(t))$ is the unique solution to the stochastic differential equation 
\begin{eqnarray}\label{example2}
\begin{cases}
dY(t)&= g(Y(t))dt  + dB(t), \quad t> 0 \\
y(0) &= \Theta(x(0))
\end{cases}
\end{eqnarray}
with 
$$g(x)= \frac{b(\Theta^{-1}(x))}{\sigma(\Theta^{-1}(x))} = (\lvert \Theta^{-1}(x) \rvert  +0.1 ) ^{p-q} .$$
Condition (1) of Hypothesis  {\textbf D} is not satisfied. However, we can apply Theorem  \eqref{teo_explosion} since equation \eqref{example2} has a unique solution  $Y$ that may explode in finite time. In consequence equation \eqref{1} has a unique solution $X$ given by 
$$
X(t)= \Theta^{-1} (Y(t)), \quad t \ge 0. 
$$
Also, \eqref{eqint_b} is satisfied.

\noindent

\section{Approximation of fractional SDEs with explosion}
\label{sec:scheme}
In this section, we introduce an adaptive Euler scheme approximating of the explosion time by following the ideas of 
\cite{acosta} and \cite{groisman}. To do so, we first consider the case that $\sigma$ is a constant diffusion, then use the Lamperti transformation to deal with the  general case. 
\subsection{Adaptive Euler scheme for constant {\boldmath $\sigma$}}
\label{sec:Euler-constant}
In this section, we approximate the explosion time of the solution to SDE
\begin{eqnarray}
\label{eqY2}
\begin{cases}
dY(t)&= g(Y(t))dt  + \sigma dB(t), \quad t> 0 \\
Y(0) &= y(0) >0,
\end{cases}
\end{eqnarray}
where $g$ is a positive, locally Lipschitz and  non-decreasing function, such that $\int_{y(0)}^\infty \frac{1}{g(s)}ds<\infty$ (and hence, by Theorem \ref{teoperalta}, $Y$ explodes in finite time), and $\sigma>0$ is a constant. Recall that $T^Y_e$ denotes the explosion time of the  solution $Y$. In order to introduce our numerical scheme, we pose an additional assumption that $g$ is 
bounded away from zero, i.e. for some constant $l_g$ we have 
\begin{equation}\label{D2}
0< l_g\leq g(y), \ \ \ \forall y\in \mathbb{R}. 
\end{equation}

Let $0<h<1$ be the fixed parameter of the numerical method.  The scheme for approximating the solution $Y$ to equation (\ref{eqY2}) is defined as 
\begin{equation}
\label{eqdefaprox1}
\begin{cases}
Y^h(t_0) &= y(0),\\
Y^h(t_{k+1})  &=Y^h{(t_k)} + g(Y^h(t_k))\tau_k+\sigma[ B(t_{k+1})-B(t_{k})],
\end{cases}
\end{equation}
where we take the random sequence $\{t_k\}$ defined by  $t_0=0$, $t_{k+1}= t_k + \tau_k$, where 
\begin{equation}
    \label{edeftau}
    \tau_k= \frac{h}{g(Y^h(t_k))}.
\end{equation}
We interpolate continuously for $t\in [t_k, t_{k+1}]$ and we write
\begin{eqnarray}
\label{eqdefaprox2}
Y^h(t)&=&Y^h{(t_k)} + g(Y^h(t_k))[t-t_k]+\sigma[ B(t)-B(t_{k})]\notag\\
&=& y(0) + \int_0^{t}g(Y^h(t_u))du+\sigma\int_0^{t}dB(u)
\end{eqnarray}
with $t_u:=t_{n_u}$, $n_u=\max\{k:t_k\leq u\}$. With $\tau_k$ given by (\ref{edeftau}), we define 
\begin{equation}
\label{edefth}
T^h_e:=\sum_{k=0}^{\infty}\tau_k = \sum_{k=0}^{\infty} \frac{h}{g(Y^h(t_k))}.
\end{equation}
Note that if $T^h_e<\infty$ almost surely, then $\tau_k$ converges to zero which further implies that $g(Y^h(t_k))$ goes to infinity, and hence $Y^h$ explodes, because $g$ is increasing. Thus $T^h_e$ corresponds to the explosion time of the numerical scheme $Y^h$. 

We will  prove the  convergence of the numerical approximations in compact (random) intervals where the solution   and the numerical approximation are bounded. 
For this result, note that the time steps $\tau_k$ are  $\mathcal{F}_{t_k}$-measurable and verify $\tau_k \leq Ch$, where $\mathcal{F}_t$ denotes the sigma-algebra generated by the driving fractional Brownian motion $B$.
Moreover, the numerical approximation $Y^{h} $ is a well-defined process
up to time $T^h_e  \wedge K $, for any $K >0$.



Now, let $M>y(0)$ be fixed. We define the following random times
\begin{equation}
\label{edefrm2}
T_{M}^Y:=\inf\{t\geq 0: |Y(t)|\geq M\}
\end{equation}
and 
\begin{equation}
\label{edefrmh}
T_{M}^{Y^h}:=\inf\{t\geq 0: |Y^h(t)|\geq M\}.
\end{equation}
By definition \eqref{edefrm}, the quantities $T_M^Y$ an $T_M^{Y^h}$ approximate explosion times $T^Y_e$ and $T^h_e$ as  $M \to \infty$. The next result gives us the convergence of $Y^h$ to $Y$ and the convergence of their respective explosion times. For the proof, see Appendix \ref{aped.A}. Let us note that the results obtained are analogous to those  presented in \cite{groisman} for the  Brownian motion case. 

\begin{theorem} 
\label{Th aproxY}

Suppose that $g$ is a strictly positive, locally Lipschitz and non-decreasing function such that  $\int_{y(0)}^\infty \frac{1}{g(s)}ds<\infty$ and \eqref{D2} holds. Let $Y$ be the solution of the fractional SDE, (\ref{eqY2}) and $Y^h$ is its approximation, as  defined by (\ref{eqdefaprox1}) and \eqref{eqdefaprox2}.  Then the  following hold:  

\begin{enumerate}
	\item[{\bf (i)}]   (Convergence of the approximation on intervals where $Y$ and $Y^h$ are finite)
Let $S > 0$ be a deterministic fixed time,  $M > y(0)$ be a constant,  and let $T_S^h=T_M^Y\wedge T_{2M}^{Y^h} \wedge S$,  where $T_M^Y $ and $T_M^{Y^h} $ are defined in (\ref{edefrm2}) and  (\ref{edefrmh}) respectively. Then, for any $p\geq 1$, 
\begin{equation*}
\lim_{h\to 0}\mathbf{E}\left(\sup_{0\leq t\leq T_S^h  }|Y(t)-Y^h(t)|^p\right)=0.
\end{equation*}
\item[{\bf (ii)}] 
\begin{enumerate} 
\item For every initial value $z > 0$, $Y^h(\cdot)$ explodes in finite time $T^h_e<\infty$ with probability one.
\item We have,
\begin{equation}\label{limYn}
\lim_{k \to \infty} \frac{Y^h(t_k)}{hk} = 1 \quad a.s.
\end{equation}
\item
Moreover, for any $\alpha > 1$ there exists $k_0 = k_0(\omega)$ such that, for every $k\ge k_0$, we have
$$
\sum_{j=k}^\infty \frac{h}{g(\alpha hj)} \leq  T^h_e - t_k \leq \sum_{j=k}^\infty \frac{h}{g(\alpha^{-1} hj)}.
$$
\end{enumerate}
\item[{\bf (iii)}] (Relation between $T_M^Y$ and $T_{M}^{Y^h}$) It holds that

$$P\left(T_M^Y>T^{Y^h}_{2M}\right) \to 0 \quad \mbox{and} \quad P\left(T_M^{Y^h}>T_{2M}^Y\right) \to 0 \quad \mbox{as} \quad {h\to 0}.$$

\item[{\bf (iv)} ] (Convergence of the approximated explosion time) For any $\varepsilon>0$ we have
$$\lim_{M\to \infty}\lim_{h\to 0}P\left(|T^{Y^h}_M- T_e^Y|>\varepsilon \right)=0.$$
\end{enumerate}
\end{theorem}

\begin{remark}
\label{remark3}
\begin{enumerate}
\item By Theorem \ref{Th aproxY} item \textbf{(i)}, we have that our adaptive method converges to the solution $Y$ of the original SDE \eqref{eqY2} on compact intervals where both $Y$ and $Y^h$ remain bounded. Actually, the statement \textbf{(i)} is a standard convergence result, see, e.g. \cite{mishura}, whenever the sample paths $Y$ of the solution to  (\ref{eqY2}) are uniformly bounded. In the case of an  explosion in finite time, which is more interesting, \textbf{(i)} simply means convergence on regions where the solution $Y$ and its approximation $Y^h$ remain bounded. In general, we do not expect convergence in larger regions. Indeed, even the statement in \textbf{(i)} does not make sense if either $Y$ or $Y^h$ are allowed to explode. 

	\item Item \textbf{(iii)} states that we can approximate the explosion time $T^Y_e$ with the sequence $T_M^{Y^h}$ as expected and as illustrated in our numerical simulations see Section \ref{sec:numerical}. As $T^{Y^h}_M$ is close to $T^h_e$ for large $M$, this also means that explosion times $T^h_e$ for the approximation $Y^h$ and $T^Y_e$ for the true solution $Y$ are close.

\end{enumerate}
\end{remark}

\subsection{Adaptive Euler scheme for non-constant $\sigma$}
In this section, we extend our adaptive Euler approximation scheme to cover the case of non-constant $\sigma$. That is, we approximate the solution $X$ to the SDE \eqref{1}. For this purpose, we require the following additional assumption on the coefficients $\sigma$ and $b$ in (\ref{1}), ensuring that the steps in  the approximating scheme are well defined.

\noindent
\textbf{Assumption E:} Coefficient functions $b$ and $\sigma$ satisfy:
\begin{enumerate}
    \item There is a positive constant $L_2$ such that
\begin{equation*}\label{D}
0< L_2\leq \sigma(x) \leq b(x), \ \ \ \ x\in \mathbb{R},
\end{equation*}
\item $b$ satisfies \eqref{eqint_b},
\item $\sigma$ satisfies $\int_{x(0)}^\infty \frac{ds}{\sigma(s)} = \infty$, that is $\Theta(\infty)=\infty$ and $\Theta(-\infty)=-\infty$.
\end{enumerate}

Let $X$ now be defined by \eqref{1}. Under condition \textbf{D}, by Theorem \ref{teo_explosion}, the solution of \eqref{1} explodes in finite time $T_e^X$ if and only if \eqref{eqint_b} holds. 
In particular, $X$ explodes in finite time under conditions \textbf{D} and \textbf{E} as Assumption \textbf{E}(3) implies $\Theta(\infty)=\infty$. Note also that by Assumption \textbf{E}(1), the function $g$ defined by \eqref{eqg} satisfies 
\begin{equation}\
0< l_g \leq 1 \leq g(x), \ \ \ \ x\in \mathbb{R}.  \nonumber
\end{equation} 
Hence the process $Y$ defined by \eqref{theta(Y)} satisfies the hypotheses of Theorem \ref{Th aproxY}. Based on the approximation $Y^h$  for solution $Y = \Theta(X)$ with $\Theta$ given by \eqref{eqTheta}, we obtain an approximation to the solution of process $X=\Theta^{-1}(Y)$ as follows:

Let $Y^h(t_{k+1})$ be defined by \eqref{eqdefaprox1} with $\tau_k$ and $t_k$ defined as in Section \ref{sec:Euler-constant}. We introduce the approximation scheme for process $X$ by

\begin{equation}
\label{eqdefXh}
\begin{cases}
X^h(t_0) &= x(0),\\
X^h(t_{k+1})  &=\Theta^{-1}(Y^h(t_{k+1})),
\end{cases}
\end{equation}
where $\Theta$ is given by \eqref{eqTheta}. The following result gives the convergence of the approximation on regions where solution $X$ and its approximation $X^h$ are bounded, in the spirit of Theorem \ref{Th aproxY} item \textbf{(i)}.

\begin{theorem} 
\label{Th 3}
Let  $\sigma $ and $b$ be functions that satisfy conditions \textbf{D} and \textbf{E}, and $S>0$  and  $M > \Theta(x(0))$ be fixed. Then, for any $p\geq 1$,
\begin{equation*}
\lim_{h\to 0}\mathbf{E}\left(\sup_{0\leq t\leq T_S^h  }|X(t)-X^h(t)|^p\right)=0,
\end{equation*}
where  $T_S^h:=T_M^Y\wedge T_{2M}^{Y^h} \wedge S$ with $T_M $ and $T_M^h $ given by (\ref{edefrm2}) and  (\ref{edefrmh}) respectively.
\end{theorem}

\begin{proof}
For $t\in [0, T_S^h]$, we have that $|Y(t)|\leq M$ and $|Y^h(t)|\leq 2M$. Since  $\Theta ^{-1}$ is  a locally Lipschitz function,

    \begin{eqnarray}
        |X(t)- X^h(t)|= |\Theta^{-1}(Y(t))-\Theta^{-1}(Y^h(t))| \leq C_M|Y(t)- Y^h(t)|.
    \end{eqnarray}
 Hence 
 $$\mathbf{E}\left(\sup_{0\leq t\leq T_S^h  }|X(t)-X^h(t)|^p\right)\leq C_M^p\mathbf{E}\left(\sup_{0\leq t\leq T_S^h  }|Y(t)-Y^h(t)|^p\right)$$
 and the result is obtained by Theorem \ref{Th aproxY} \textbf{(i)}.

 \qed

\end{proof}

\begin{remark}
\label{Re time2}
It follows directly from the transformations $X= \Theta^{-1}(Y)$ and $X^h = \Theta^{-1}(Y^h)$ that the explosion times for true solutions $X$ and $Y$ and the approximated solutions $X^h$ and $Y^h$ are equal. As such, one obtains Theorem \ref{Th aproxY} \textbf{(iii)} for $X$ and $X^h$ directly, allowing us to approximate the explosion time $T_e^X$ with $T_M^{X^h}$, cf. Remark \ref{remark3}. 
\end{remark}

\section{Numerical examples}
\label{sec:numerical}
We illustrate our adaptive numerical scheme with two simulation studies related to the examples in Section \ref{sec:examples}. For effective simulations, we  use of the prediction formula of the fractional Brownian motion presented in \cite{So-Vi}. That is, the future law of $B(t)$ given information up to $u<t$ is Gaussian with (random) mean
$$
B(u) - \int_0^u \Psi(t,s|u)dB(s),
$$
where 
$$
\Psi(t,s|u) = -\frac{\sin(\pi(H-1/2))}{\pi}s^{1/2-H}(u-s)^{1/2-H}\int_u^t \frac{z^{H-1/2}(z-u)^{H-1/2}}{z-s}dz,
$$
and a deterministic covariance
$$
r(t,s|u) =R(t,s)-\int_0^u K(t,v)K(s,v)dv
$$
with $R$ being the covariance of $B$
and $K$ the Volterra Kernel of the fractional Brownian motion given by 
$$
K(t,s)=d_H \left[  \left(  \frac{t}{s} \right)^{H-1/2} (t-s)^{H-1/2} -(H-1/2) s^{1/2-H} \int_s^t z^{H-3/2} (z-s)^{H-1/2} dz \right].
$$
Consequently, given the information up to $t_k$ (in which case we know $t_{k+1}$ as  well), we have, with random times $t_k$ and $t_{k+1}$,
\begin{equation}\label{CondNormal}
    B(t_{k+1})-B(t_k) \sim N(\mu,\sigma^2)
\end{equation}
with 
$$
\mu = \int_0^{t_k} \Psi(t_{k+1},s|t_k)dB(s), \quad \mbox{and} \quad 
\sigma^2 = r(t_{k+1},t_{k+1}|t_k).
$$
This allows us to generate increments $B(t_{k+1})-B(t_k)$ as in \eqref{CondNormal} by using MATLAB.

{\bf Example 1:} The first example  is given as in equation \eqref{ejemplo1}, defined via  
$$
dX(t)= X^N(t) dt + X(t) dB(t)
$$
with $N=4$, $H=0.65$, and $x(0)=10$; $Y(0)=0$. We set $h=0.1$.
\begin{figure}[!h!tb]
    \centering
    \includegraphics[width=10cm, height=6cm] {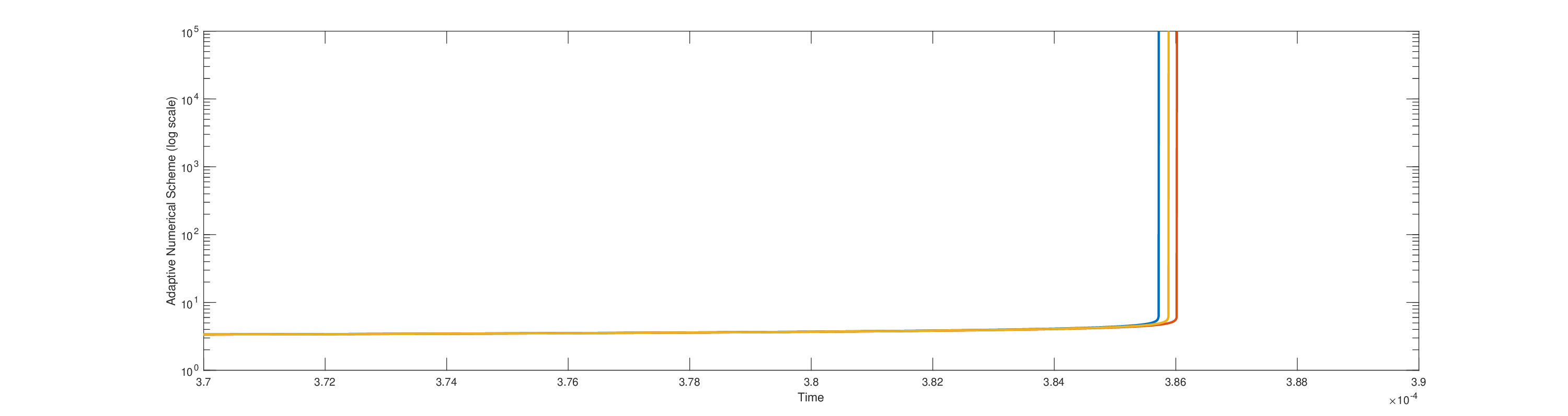}
    \caption{Adaptive Euler scheme: Sample paths with explosions for the solution of \eqref{ejemplo1}.}
    \label{fig:SP-1-1}
\end{figure}
Figure \ref{fig:SP-1-1} shows few simulated sample paths (in the log-scale) of a solution to \eqref{ejemplo1} from which explosion is clearly visible. We also note that the explosion times are relatively small. This is expected due to the exponential growth of the drift term that leads to rapid explosion. 
  \begin{figure}[!h!b]
         \centering
    \includegraphics[width=10cm, height=6cm] {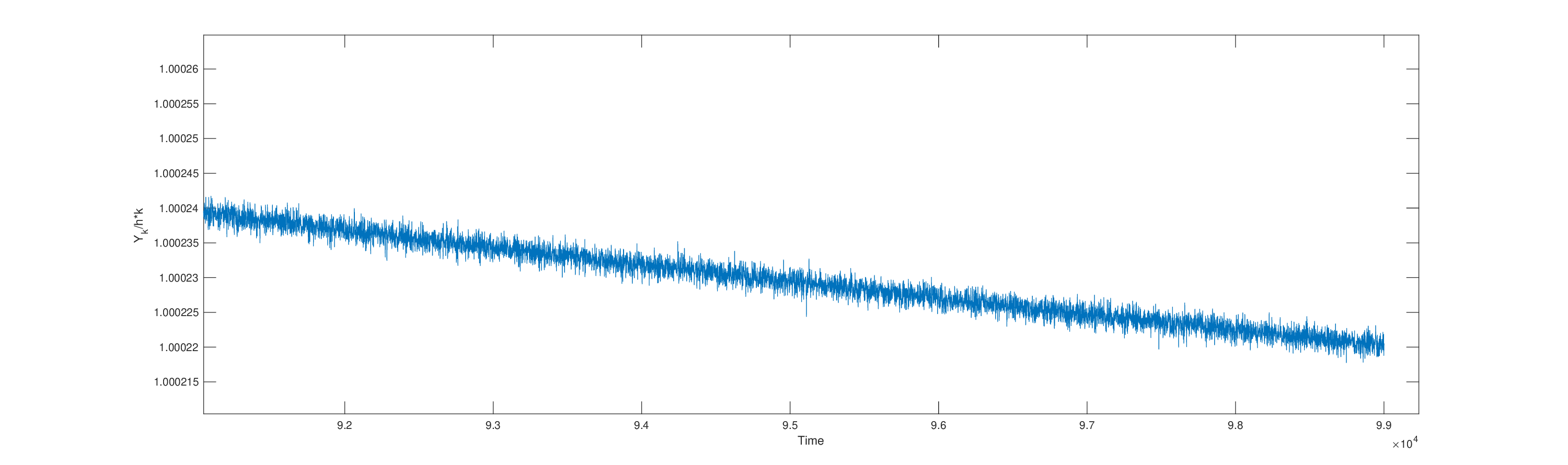}
        \caption{Convergence of $\frac{Y^h(t_k)}{hk}$ to 1 related to Equation \eqref{ejemplo1}.}
       \label{fig:AS-1-1}
    \end{figure}
Figure \ref{fig:AS-1-1} shows, as demonstrated in theorem \ref{Th aproxY} (b), the convergence of $\frac{Y^h(t_k)}{hk}$ to  $1$ a.s., as $k \to \infty$. Note, in particular, the scales of the $x$- and $y$-axes. 
  \begin{figure}[!h!b]
        \centering
        \includegraphics[width=10cm, height=6cm] {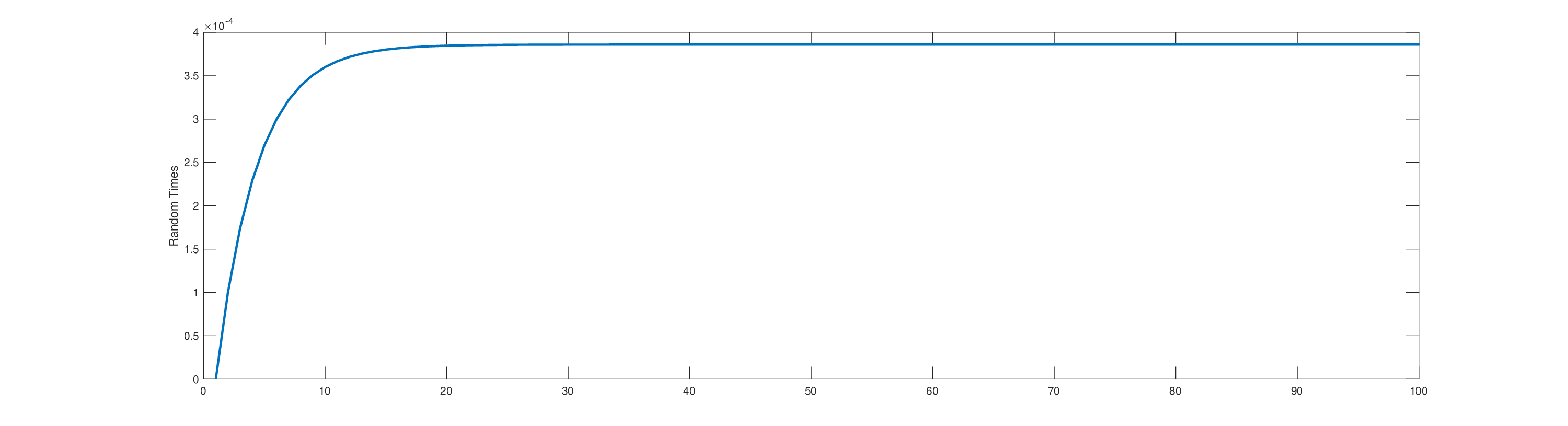}
        \caption{Random times approximating explosion time.}
       \label{fig:2AS-1-1}
    \end{figure}
  Figure \ref{fig:2AS-1-1} shows the random times $t_k$ approximating the explosion time. In Figure \ref{fig:2AS-1-1}, the number of time steps is represented in the $x$-axis. Finally, Figure  \ref{fig:Ex1-zoom+as-1} illustrates the explosion of $Y$ and $X$ via single trajectories. Observe that, see Figure \ref{fig:Ex1-FBM-1}, that the increments of the fractional Brownian motion are relatively small. This is because the conditional variance becomes smaller since differences $t_{k+1}-t_k$ are small (due to rapid explosion). 
\begin{figure}[!h!tb]
 \begin{subfigure}[!h!tb]{0.7\textwidth}
  \centering
         \includegraphics[width=0.7\textwidth]{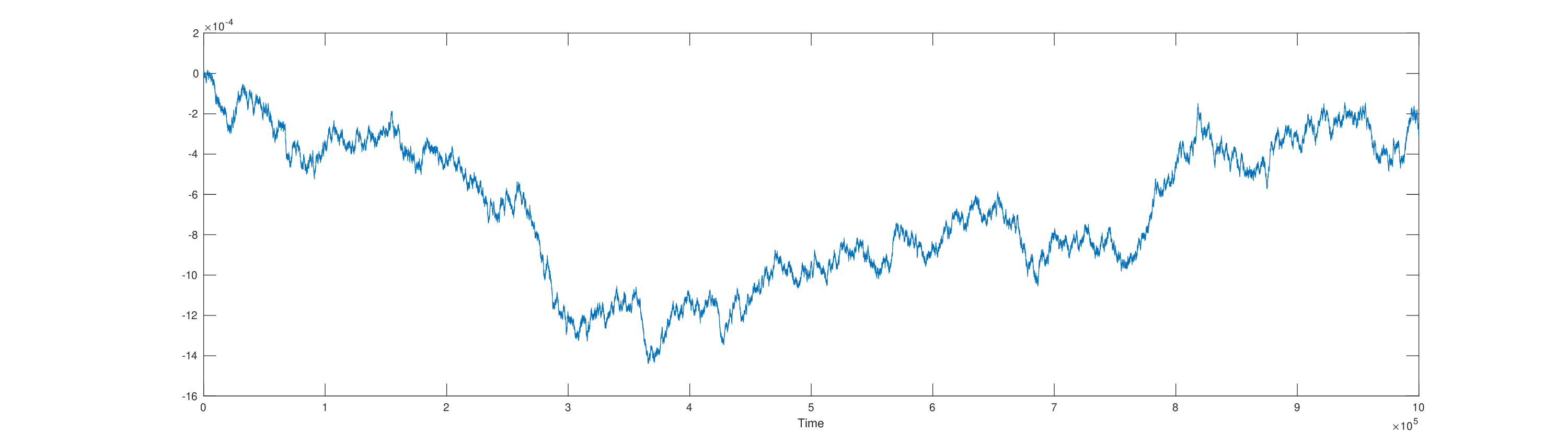}
        \caption{The zoomed trajectory of a single path from Figure  \ref{fig:SP-1-1}.}
         \label{fig:Ex1-FBM-1}
     \end{subfigure}
\begin{subfigure}[!h!tb]{0.7\textwidth}
         \centering
         \includegraphics[width=0.7\textwidth]{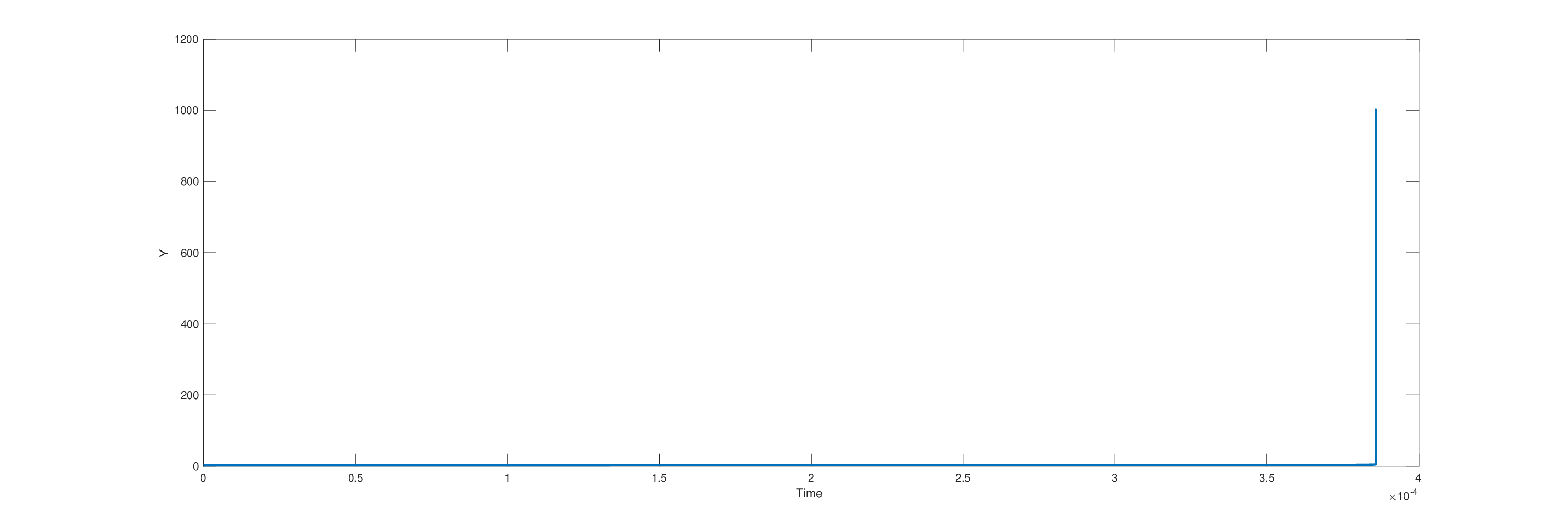}
         \caption{One trajectory of solution $Y$.}
         \label{fig:Ex1-Nuevo}
     \end{subfigure}   
     \begin{subfigure}[!h!tb]{0.7\textwidth}
         \centering
         \includegraphics[width=0.7\textwidth]{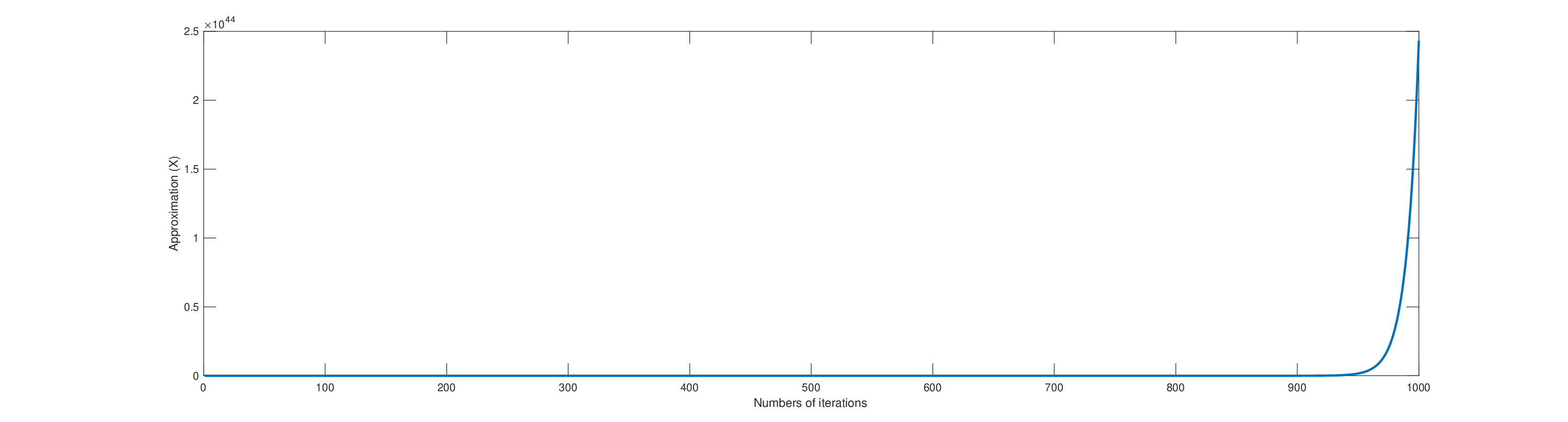}
         \caption{One trajectory of solution $X$.}
         \label{fig:trayectoryX-1}
     \end{subfigure}
       \caption{The behavior of one of the $X$ and $Y$ path related to Equation \eqref{ejemplo1}.}
       \label{fig:Ex1-zoom+as-1}
\end{figure}
\newpage

\vspace{1cm}
{\bf Example 2:} The second example  is given as in equation \eqref{example2} with parameters 
$H=0.65$; $h=0.1$; $p=1.1$, $x(0)=10$, and  $q=0.5$. Figure \ref{fig:SP-2}  shows few simulated sample paths for solution $Y$ (stopped when the computed solution reaches $10^5$) from which explosion is clearly visible, while convergence of $\frac{Y^h(t_k)}{hk}$ towards 1 is illustrated in Figure \ref{fig:AS-1}.
\begin{figure}[!h!tb]
    \centering
    \includegraphics[width=10cm, height=6cm] {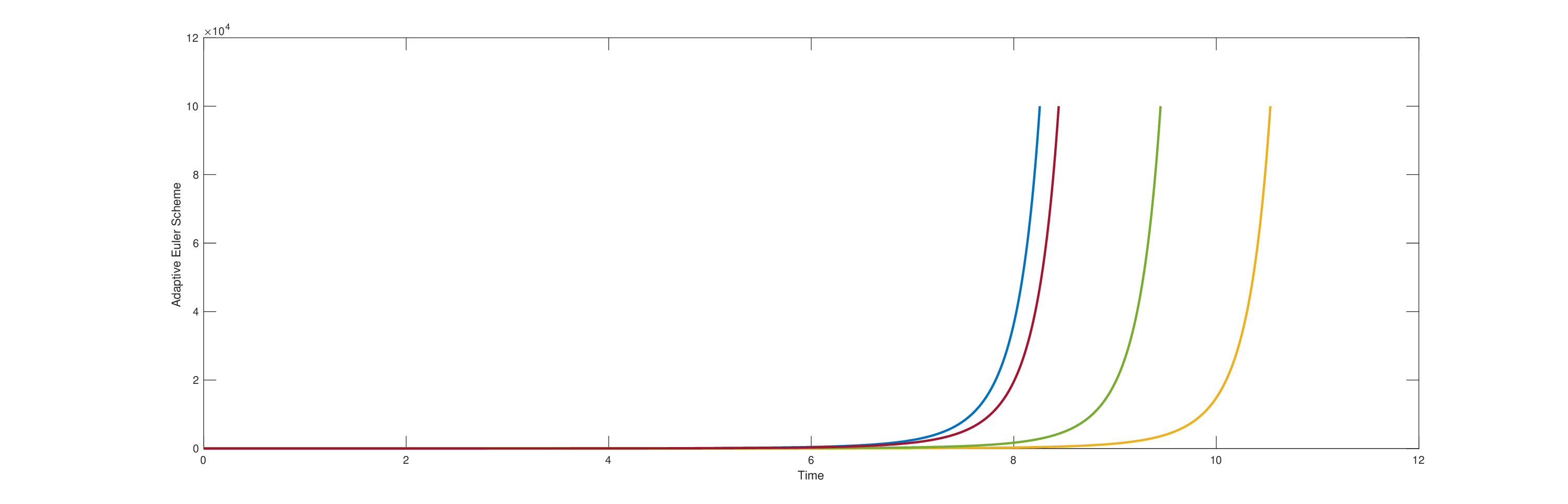}
    \caption{Adaptive Euler scheme: sample paths with explosion for the solution of \eqref{example2}.}
    \label{fig:SP-2}
\end{figure}
\begin{figure}[!h!tb]
         \centering
         \includegraphics[width=10cm, height=6cm] {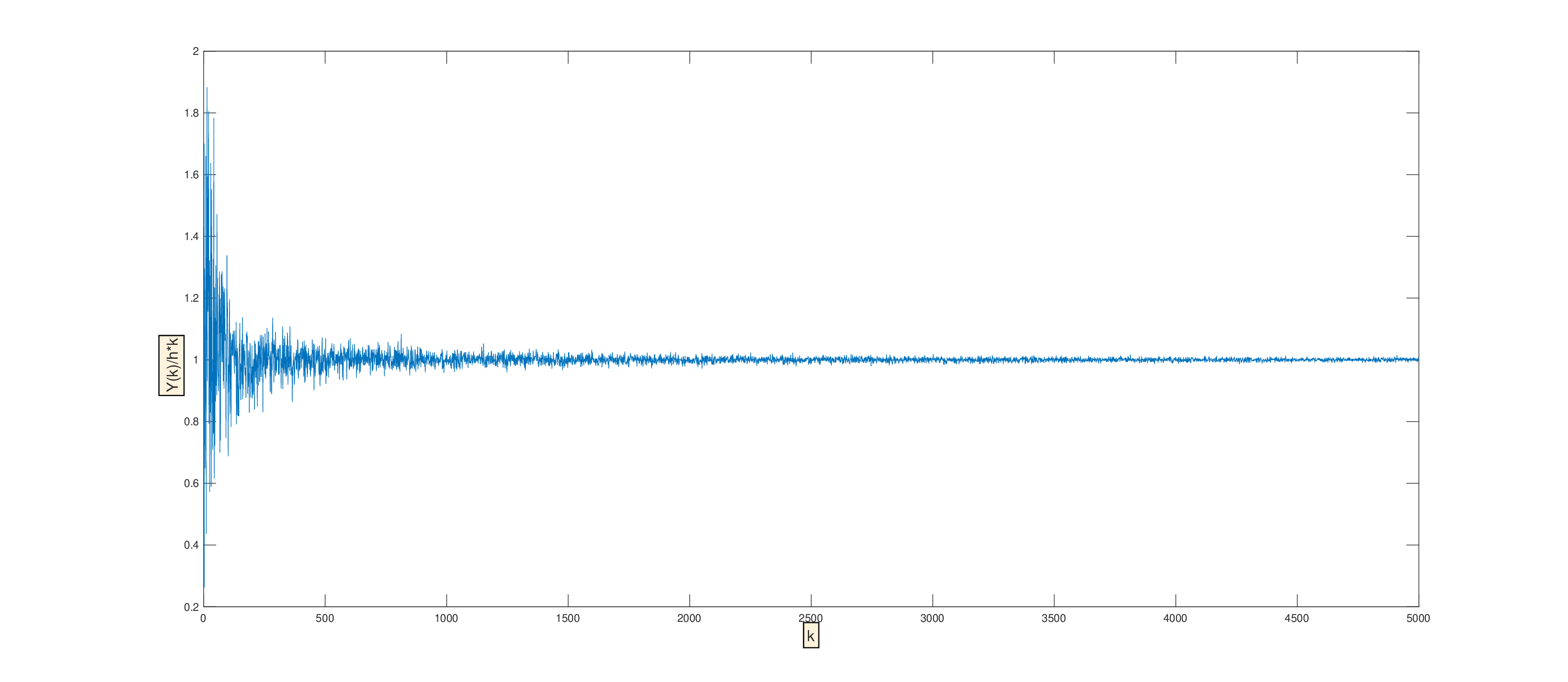}
         \caption{Convergence of $\frac{Y^h(t_k)}{hk}$ to 1 related to Equation \eqref{example2}.}
         \label{fig:AS-1}
     \end{figure}
Figure \ref{fig:Random} shows the random times approximating the explosion time, and finally, Figure  \ref{fig:Ex2-zoom+asXY} illustrates the  explosion of $Y$ and $X$ via single trajectories. In comparison to the previous example, note the explosion in this case does not happen as rapidly, as the drift term does not grow as fast. 
\begin{figure}[h!]
         \centering
         \includegraphics[width=10cm, height=6cm] {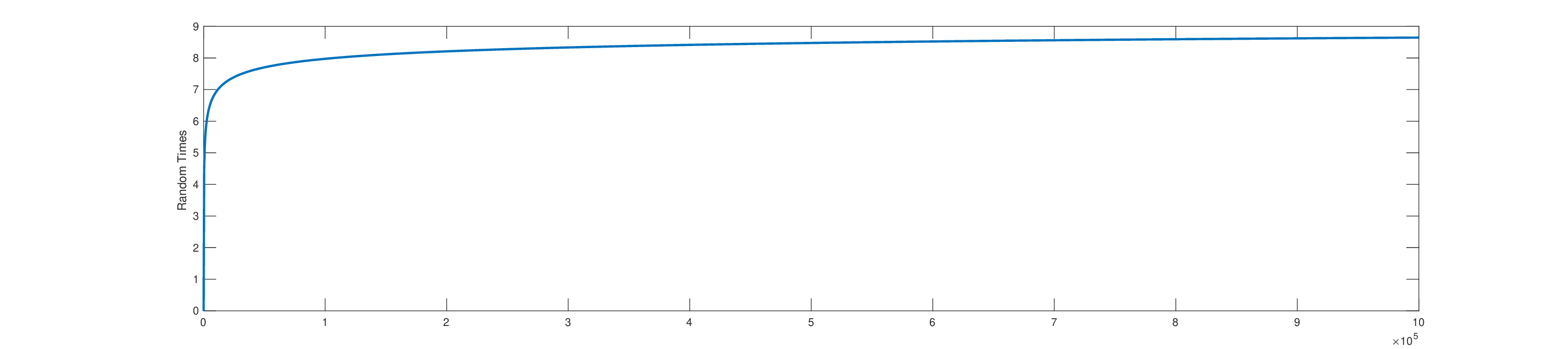}
         \caption{Random times approximating explosion time in Equation \eqref{example2}.}
         \label{fig:Random}
     \end{figure}
\begin{figure}[!h!tb]
     \begin{subfigure}[!h!tb]{0.7\textwidth}
         \centering
         \includegraphics[width=0.7\textwidth]{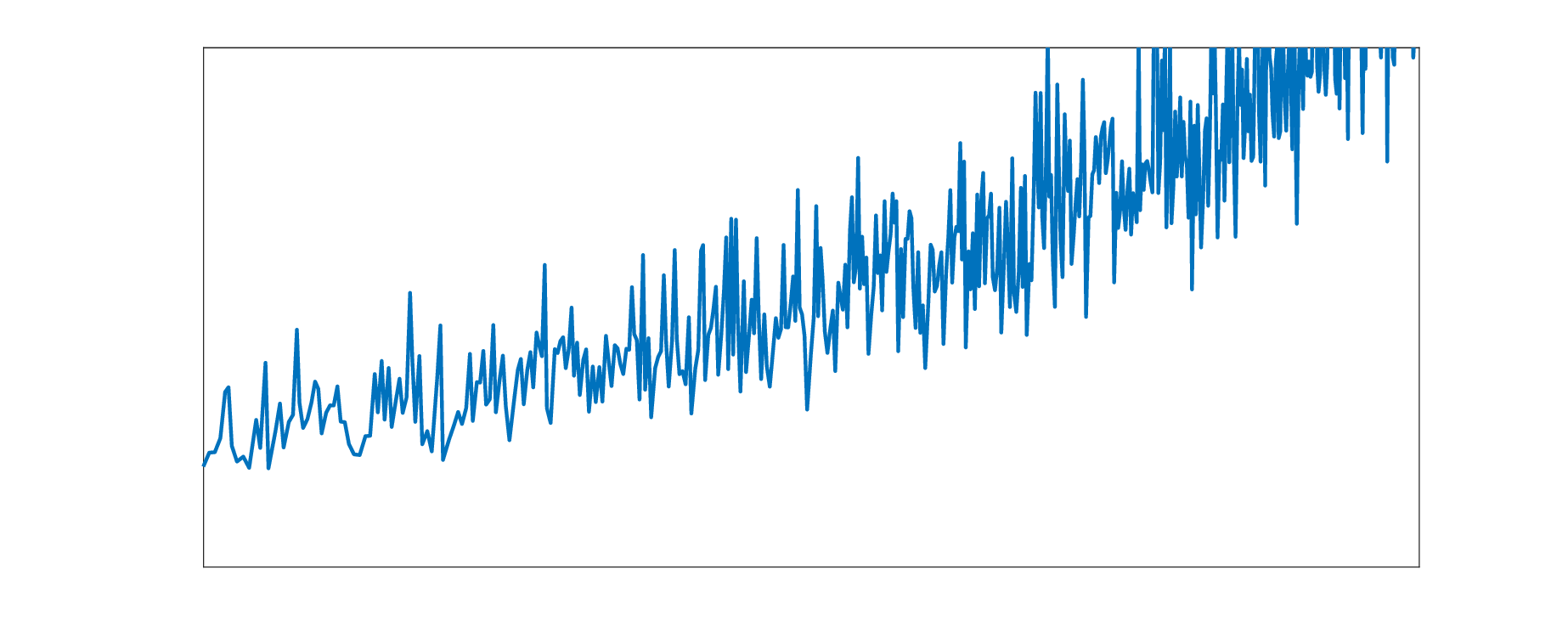}
         \caption{The zoomed trajectory of a single path from Figure  \ref{fig:SP-2}.}
         \label{fig:Ex2-zoom}
     \end{subfigure}
        \label{fig:Ex2-zoom+as}
     \begin{subfigure}[!h!tb]{0.7\textwidth}
         \centering
         \includegraphics[width=0.7\textwidth]{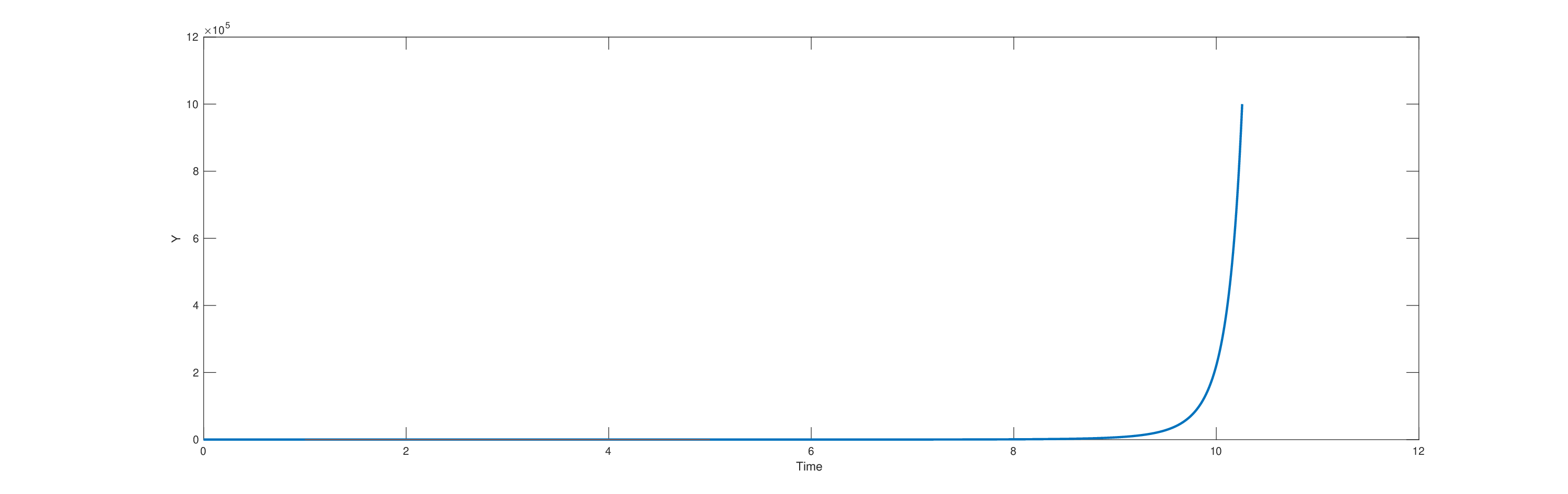}
         \caption{One trajectory of solution $Y$.}
         \label{fig:Ex2-Nuevo}
     \end{subfigure}   
     \begin{subfigure}[!h!tb]{0.7\textwidth}
         \centering
         \includegraphics[width=0.7\textwidth]{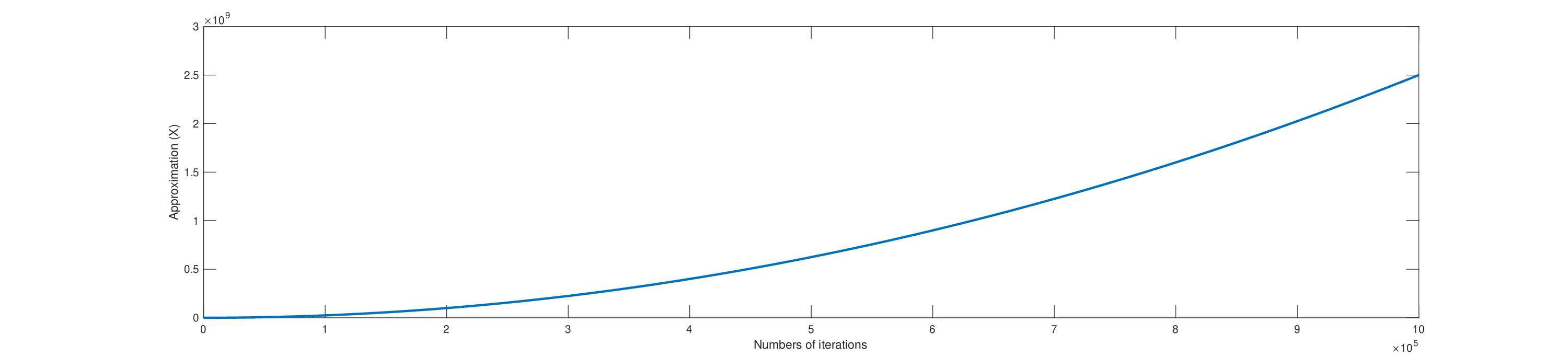}
         \caption{One trajectory of solution $X$.}
         \label{fig:trayectoryX}
     \end{subfigure}
        \caption{The behavior of one $X$ and $Y$ path related to Equation \eqref{example2}.}
        \label{fig:Ex2-zoom+asXY}
\end{figure}

 \newpage


{\bf Acknowledgments:} 
Johanna Garz\'on was partially supported by  HERMES project 58557.
Soledad
Torres was partially supported by Fondecyt project number 1230807 and Mathamsud
SMILE AMSUD230032,Proyecto ECOS210037 (C21E07),
Mathamsud AMSUD210023 and Fondecyt Regular No. 1221373.
The authors acknowledge the support of Centro de Modelamiento Matemático (CMM)
FB210005, BASAL funds for centers of excellence from ANID-Chile.

\appendix 

\section{Proof of Theorem \ref{Th aproxY}}
\label{aped.A}
To prove Theorem \ref{Th aproxY} (i), we will first give a priori bounds on our adaptive numerical scheme that controls  the  increments of the approximation $Y^h$ defined in (\ref{eqdefaprox1}) and the convergence of the adaptive scheme $Y^h$ to $Y$ when we are working with a bounded function $g$. The arguments are similar to the arguments presented in \cite{mishura}. However, due to the presence of random times instead of deterministic times in our approximation method, we prefer to provide the key steps  for the reader's convenience.  

\begin{lemma}
\label{tcotay}
Suppose that $g$ is a positive, locally Lipschitz and non-decreasing function such that for some $L_g, l_g>0$,
$$0<l_g\leq g(x) \leq L_g, \forall x. $$

Then for any $0<\rho<H$ and any $S>0$ there exists a non-negative random variable $F_{{Y^h}}$  such that $\mathbf{E}(F_{{Y^h}}^p)< \infty$ for all $p>1$ and
\begin{equation}
\label{cota priori}
|Y^h(s)-Y^h(u)|\leq F_{Y^h}(s-u)^{H-\rho}, \ \ 0\leq u<s\leq S. 
\end{equation}
Moreover, it holds that
 $$\sup_{0\leq t\leq S}|Y(t)-Y^h(t)|\leq F h^{2H-1-\rho}$$ 
 \\
for a random variable $F$ independent of $h$ and satisfying $\mathbf{E}(F^p)<\infty$ for all $p\geq 1$.
\end{lemma}

\begin{proof}
From Lemma \ref{tckolmogorov} and \eqref{eqdefaprox2} we deduce
\begin{eqnarray*}
\label{eqcotay-y}
|Y^h(s)-Y^h(u)|&\leq& \int_u^{s} |g(Y^h(t_v))|dv + \sigma \left| \int_u^{s} dB(v) \right|\notag\\
&\leq & L_g(s-u) + \sigma  |B(s) -B(u)| \notag\\
& \leq & L_gS^{1-H+\rho}(s-u)^{H-\rho} + \sigma  F_{B}(s-u)^{H-\rho}\notag\\
& \leq &  F_{Y^h}(s-u)^{H-\rho},
    \end{eqnarray*}
    where $ F_{Y^h}=L_gS^{1-H+\rho}+ \sigma F_{B}$ and hence $\mathbf{E}(F_{Y^h}^p)<\infty$ for all $p\geq 1$.
For the other claim, we follow the proof of Theorem 3.1 in \cite{mishura}. Indeed, by noting that in \cite{mishura} the arguments hold almost surely, and hence can be adapted for the random times $t_k$. 
Hence a priori bound \eqref{cota priori} leads to 
 $$\sup_{0\leq t\leq S}|Y(t)-Y^h(t)|\leq F h^{2H-1-\rho},$$
where $F$ is a random variable such that $\mathbf{E}(F^p)<\infty$ for all $p\geq 1$ and $F$  does not depend on $h$. This completes the proof. 
\qed
\end{proof}

\

\b  {\bf Proof of Theorem \ref{Th aproxY} (i): } We define the globally Lipschitz and bounded function  $\bar{g}$ :
\begin{equation}
    \label{edefb}
    \bar{g}(x)=
\begin{cases}
    g(x) & \quad \text{if} \quad |x|\leq 2M, \\
    g(2M) & \quad \text{if} \quad x\geq 2M,\\
    g(-2M) & \quad\text{if} \quad x\leq -2M.
\end{cases}
\end{equation}
From Theorem \ref{texistenciax}, and $S>0$, there exists a unique solution $Z$, of the equation 
 \begin{equation}
    \label{edefZ}
    dZ(t)= \bar{g}(Z(t))dt + \sigma dB(t), \quad \quad Z(0)=y(0), \ t\in[0, S],
\end{equation}
and this solution $Z$ satisfies, for any $0<\rho<H$ and almost surely,  \begin{equation}
\label{ecotaz}
\sup_{0\leq s\leq t \leq S}|Z(t)-Z(s)|< F_{Z}|t-s|^{H-\rho},
\end{equation}
where $F_{Z}=F_{Z} (\omega)$ satisfies $\mathbf{E}(F_{Z}^p)<\infty$ for all $p\geq 1.$

\

 Let $Z^h$ now be the approximation given by 
 \begin{equation}
    \label{edefZtilda}
{Z}^h(t_{n+1})  ={Z}^h{(t_n)} + \bar{g}({Z}^h(t_n))\tau_{n} + [ B(t_{n+1})-B(t_{n})], \quad \quad {Z}(0)=y(0).
\end{equation}
As $g(x)=\bar{g}(x)$  for $|x|\leq 2M$ and $T^h_S=T^{Z}_M\wedge T_{2M}^{Z^h} \wedge S$, we have $Y(t)=Z(t)$ and ${Y}^h(t)={Z}^h(t)$ for $t\in [0, T_S^h]$, and thus
$$\mathbf{E}\left(\sup_{0\leq t \leq T_S^h}|Y(t)-Y^h(t)|^p\right)=\mathbf{E}\left(\sup_{0\leq t \leq T_S^h}|Z(t)-Z^h(t))|^p\right)\leq \mathbf{E}\left(\sup_{0\leq t \leq S}|Z(t)-Z^h(t)|^p\right)$$
from which the claim followed by Lemma \ref{tcotay} with $\rho < 2H-1$  applied to the process $Z$ and its approximation $Z^h$.
\qed

\vspace{0.3cm}

The rest of the proof of Theorem  \ref{Th aproxY} is based on the following lemma: In contrast to the Brownian motion, we do not have martingale property at our disposal; 
 instead, we rely on the law of the iterated logarithm for  fBm.

\begin{lemma}   Let $\{t_k\}$ be defined by $t_0=0$, $t_{k+1}= t_k + \tau_k$, where $\tau_k$ is given by 
$\tau_k= \frac{h}{g(Y^h(t_k))}$ (see \eqref{edeftau}).  Then
\label{laux}
$$\lim_{k\to \infty} \frac{  B({t_k})  }{k}=0 \quad \text{a.s.}
$$
\end{lemma}
\begin{proof} Let us consider 
	$$\Omega = \displaystyle  \bigcup_{k=1}^\infty \quad [ \{\sup_{l \in \mathbb{N}}  t_l \} \leq k] \quad \bigcup \quad [ \{\sup_{l \in \mathbb{N}} t_l \}  = \infty ].$$
Let us first consider the case $\omega \in \displaystyle  \bigcup_{k=1}^\infty \quad [ \{\sup_{l \in \mathbb{N}}  t_l^n(\cdot) \} \leq k] $ (corresponding to the case of the explosion), which is easy to handle. Indeed, in this case, there exists $k \in \mathbb{N}$ such that 
\begin{eqnarray}\label{INE-2-1}
	\frac {|B({t_l(\omega)},\omega)|} {l}   
	&\leq&  \frac {C(\omega , k) k^{H - \delta}}{l} 
\end{eqnarray}
where $C(\omega , k)$ depends only on $k$, $\omega$ and $\delta < H$. 
Consequently, 
$$
\lim_{l \to \infty} \frac {|B({t_l(\omega)},\omega)|} {l}  =0.
$$
Suppose now the case that $\omega  \in \left[ \{\displaystyle \sup_{l \in \mathbb{N}}  t_l \} = \infty \right] $. We recall that for $H\in (0, 1)$ the fBm satisfies the following law of the iterated logarithm (see Lemma 4.1 in \cite{leon}).
	\begin{equation}\label{LIL-fBm}
	\limsup_{t \to \infty}	\frac {B(t)} {\psi_H(t) }
	= 1 \quad \text{a.s.}, 
\end{equation}
 where we use the notation
\begin{equation}
	\psi_H(t) := t^H \sqrt{2 \log \log t}, \quad  t > e^1.
	\end{equation}
Since  the sequence $\{t_l\}_{l \in \mathbb{N}}$  is increasing to infinity, then from \eqref{LIL-fBm}, there exists $\tilde{l} > 0$ such that for all $l > \tilde{l}$ 
\begin{equation}
		\frac {B({t_l(\omega)},\omega)} { t_l^H(\omega) \sqrt{2 \log \log t_l(\omega)}}
	\leq  1+\epsilon < 2
\end{equation}
almost surely. Moreover, for large enough $l$ (depending on $\omega$), we also have $\log(t_l(\omega)) > 1$. This, together with the fact that $-B$ is also a fBm, gives 
\begin{eqnarray}\label{INE-1}
		\frac {|B({t_l(\omega)},\omega)|} {l}   
		&=&  \frac {|B({t_l(\omega)},\omega)|} { t_l^H(\omega) \sqrt{2 \log \log t_l(\omega)}} \times  \frac{ t_l^H(\omega) \sqrt{2 \log \log t_l(\omega)}}{l} \nonumber \\
		&\leq & 2  \frac{ t_l^H(\omega) \sqrt{2 \log \log t_l(\omega)}}{l}. 
\end{eqnarray}
From the definition \eqref{edeftau} and  hypothesis \eqref{D2}, $t_l(\omega) \leq C_1 l$ for some constant $C_1 > h / l_g$ independent  of $l$. Then \eqref{INE-1} becomes\begin{eqnarray}\label{INE-2}
	\frac {|B({t_l(\omega)},\omega)|} {l}   
	&\leq &  \frac{2 C^H_1 l^H \sqrt{2 \log \log t_l(\omega)}}{l} 
	\leq   \frac{2 C^H_1 \sqrt{2 \log \log (C_1l)}}{l^{1-H}}. 
\end{eqnarray}
The claim now follows since the right-hand side converges to zero as $l\to \infty$. This completes the proof.

\qed
\end{proof}

\

\b {\bf Proof of Theorem \ref{Th aproxY} (ii)-(iv): } 

For the claim (ii), it suffices to use Lemma \ref{laux} and follows the proof  \cite[Theorem 1.3]{groisman}. Similarly, for claim (iii), we follow proof  \cite[Proposition 3.2]{groisman}, from which the claim (iv) followed by the arguments of  proof  \cite[Theorem 1.4]{groisman}. This completes the proof.
\qed

\end{document}